\documentclass[12pt]{article}

\newlength{\BiblioSpacing}
\renewenvironment{thebibliography}[1]{%
  \begin{oldthebibliography}{#1}%
    \setlength{\parskip}{\BiblioSpacing}
    \setlength{\itemsep}{\BiblioSpacing}
}%
{%
\end{oldthebibliography}%
}

\usepackage{latexsym}
\usepackage{amssymb}
\usepackage{amsfonts}
\usepackage{amsmath}
\usepackage{indentfirst}
\usepackage{epsfig}

\usepackage{epic, eepic}

\newtheorem{thm}{Theorem}[section]
\newtheorem{coro}[thm]{Corollary}

\newtheorem{conj}[thm]{Conjecture}
\newtheorem{lemma}[thm]{Lemma}

\newtheorem{defn}[thm]{Definition}

\newcommand{\eproof}{\hspace*{\fill}$\Box$\vspace{1mm}}
\newenvironment{proof}{\noindent \emph{Proof.}}{\eproof}

\newcommand\oeis{{\small\sc OEIS}}

\newcommand\ol{\overline} 

\newcommand\seta{\ensuremath{\mathcal S}}
\newcommand\setpn{\ensuremath{\seta_p(n)}}
\newcommand\setp{\ensuremath{\seta_p}}
\newcommand\numa{\ensuremath{A}} 
\newcommand\numpn{\ensuremath{\numa_p(n)}}

\DeclareMathOperator{\asc}{\mathrm{asc}}
\DeclareMathOperator{\des}{\mathrm{des}}
\DeclareMathOperator{\fwd}{\mathrm{{\sc fwd}}}
\DeclareMathOperator{\zeros}{\mathrm{zeros}}
\DeclareMathOperator{\rlmax}{{\mathrm{\mbox{\small\sc RL}max}}}
\DeclareMathOperator{\rlmin}{{\mathrm{\mbox{\small\sc RL}min}}}
\DeclareMathOperator{\lrmax}{{\mathrm{\mbox{\small\sc LR}max}}}
\DeclareMathOperator{\lrmin}{{\mathrm{\mbox{\small\sc LR}min}}}

\def\dd{\makebox[1.1ex]{\rule[.58ex]{.71ex}{.15ex}}}

\newcommand\bx{\ensuremath{\mathbf{x}}}
\newcommand\by{\ensuremath{\mathbf{y}}}
\newcommand\bt{\ensuremath{\mathbf{t}}}

\def\ch#1,#2,{{#1\choose#2}}

\def\clp{{\mathcal P}} 
\def\clr{{\mathcal R}}

\def\om{\ensuremath{\omega}}

\title{Pattern avoidance in ascent sequences}

\author{
  Paul Duncan\\
  \small 25 Vega Court \\
  \small Irvine, CA 92617, USA\\
  \small \texttt{pauldncn@gmail.com}
  \and Einar Steingr\'imsson\thanks{Steingr\'imsson was supported by
    grant no.\ 090038013 from the
    Icelandic Research Fund.}\\
  \small Department of Computer and Information Sciences\\
  \small University of Strathclyde, Glasgow G1 1XH, UK\\
  \small \texttt{einar.steingrimsson@cis.strath.ac.uk}
}

\usepackage[usenames]{color}
\usepackage{colordvi}

\begin{document}
\maketitle

\small Mathematics Subject Classification: 05A05, 05A15, 05A18, 05A19

\begin{abstract}
Ascent sequences are sequences of nonnegative integers with
restrictions on the size of each letter, depending on the number of
ascents preceding it in the sequence.  Ascent sequences have recently
been related to $(2+2)$-free posets and various other combinatorial
structures.  We study pattern avoidance in ascent sequences, giving
several results for patterns of lengths up to 4, for Wilf equivalence
and for growth rates.  We establish bijective connections between
pattern avoiding ascent sequences and various other combinatorial
objects, in particular with set partitions.  We also make a number of
conjectures related to all of these aspects.
\end{abstract}

\thispagestyle{empty}

\section{Introduction and preliminaries}

An \emph{ascent sequence} is a sequence $x_1x_2\ldots x_n$ of nonnegative integers satisfying $x_1=0$ and, for all $i$ with $1<i\le n$,
$$ x_i\le\asc(x_1x_2\ldots x_{i-1})+1,
$$ where $\asc(x_1x_2\ldots x_k)$ is the number of \emph{ascents} in
the sequence $x_1x_2\ldots x_k$, that is, the number of places $j\ge1$
such that $x_{j}<x_{j+1}$.  An example of such a sequence is
0101312052, whereas 0012143 is not, because the 4 is greater than
$\asc(00121)+1=3$.  Replacing $x_j<x_{j+1}$ in the definition of
$\asc$ by $x_j>x_{j+1}$ gives the number of \emph{descents}.

Ascent sequences became prominent after they were related to the
$(2+2)$-free posets, by Bousquet-M\'elou, Claesson, Dukes and Kitaev
\cite{bcdk}, who also managed to find the generating function counting
these, which was quite a feat.  Ascent sequences have since been
studied in a series of papers by various authors, connecting them to
many other combinatorial structures.  These connections, and
generalizations of them, have exposed what seem to be deep structural
correspondences in these apparently disparate combinatorial objects,
including certain integer matrices, set partitions and permutations,
in addition to the $(2+2)$-free posets.  A good source of references
and further information is \cite[Section~3.2.2]{kit-book}; see also
\cite{claesson-linusson, dkrs, dukes-parviainen,
  kitaev-remmel-2plus2}.

In this paper we study ascent sequences avoiding certain patterns.
Our patterns are analogues of permutation patterns, but they seem to
provide a greater variety of counting sequences, which may perhaps be
explained by the fact that ascent sequences lack the so called trivial
symmetries on permutations---such as reversing a permutation to obtain
another one---that imply numerical equivalence between sets of
permutations avoiding different patterns.

A pattern is a word on nonnegative integers, where repetitions are
allowed.  An occurrence of a pattern $p$ in an ascent sequence
$\bx=x_1x_2\ldots x_n$ is a subsequence $x_{i_1}x_{i_2}\ldots x_{i_k}$
in $\bx$, where $k$ equals the length of $p$, whose letters appear in
the same relative order of size as those in $p$.  For example, the
ascent sequence 0123123 has three occurrences of the pattern 001,
namely in the subsequences 112, 113 and 223.  Note that in an
occurrence of a pattern~$p$, letters corresponding to two equal
letters in $p$ must be equal in the occurrence, such as the 22 in 223,
which correspond to the two 0's in 001.  An ascent sequence $\bx$
\emph{avoids} the pattern $p$ if $\bx$ has no occurrences of $p$.  As
an example, 012321 avoids 001.


Since we write our patterns with nonnegative integers, whereas
permutation patterns have traditionally been written with positive
integers, it is important to note that the traditional permutation
patterns have different names here.  For example, 123 becomes 012, and
231 becomes 120.  We use the latter notation, with nonnegative rather
than positive integers, since ascent sequences are traditionally
defined in such a way as to contain zeros.

To the best of our knowledge, pattern avoidance has not been studied
for ascent sequences so far. Given the success of such studies for
other combinatorial structures, such as permutations and set
partitions, and the strong connections of ascent sequences to other
combinatorial objects mentioned above, it is reasonable to hope for
results of similar significance in the case of ascent sequences.  The
initial results and conjectures presented here, together with the
great variety in the integer sequences counting ascent sequences
avoiding various patterns, seem to indicate that this is fertile
ground for interesting research.

We exhibit a connection between ascent sequences and set partitions,
which does not seem to have been studied before, although such a
connection, of a very different nature, is considered in
\cite{dukes-parviainen}.  A standard way to represent set partitions
of the set $\{1,2,\ldots n\}$, or any other ordered set, is to write
the elements of each block in increasing order, and the blocks in
order of increasing minima, such as in 124-36-5, which represents the
partition of $\{1,2,3,4,5,6\}$ into blocks $\{1,2,4\}$, $\{3,6\}$ and
$\{5\}$.  This standard representation of a partition of $\{1,2,\ldots
n\}$ can be encoded by the string $a_1a_2\ldots a_n$, where $a_i=k$ if
$i$ belongs to block number $k$, counting from left to right, with the
leftmost block numbered 0.  In our previous example of $124\dd36\dd5$
this string would be 001021.  It is easy to see that a string $\bx$ of
nonnegative integers encodes a set partition if and only if the first
occurrence of each letter $k>0$ in $\bx$ is preceded by some
occurrence of $k-1$.  Such a string is called a \emph{restricted
  growth function}, or RGF.  A \emph{non-crossing partition} is a
partition that does not have letters $a<b<c<d$ with $a,c$ in one block
and $b,d$ in another.

If $p$ is a pattern, we let $\setpn$ be the set of ascent sequences of
length $n$ that avoid $p$ and $\numpn$ the number of such sequences.
Also, $\setp$ is the union of $\setpn$ for all $n\ge1$.  The set of
\emph{left-to-right maxima} in a sequence of numbers $a_1a_2\ldots
a_n$ is the set of $a_i$ such that $a_i>a_j$ for all $j<i$.  The set
of \emph{right-to-left minima} is the set of $a_i$ such that $a_i<a_j$
for all $j>i$. \emph{Left-to-right minima} and \emph{right-to-left
  maxima} are defined analogously.  Also, we define $\lrmax(\bx)$ to
be the number of left-to-right maxima in a sequence $\bx$, and
$\lrmin$, $\rlmax$, $\rlmin$ analogously.

Following is an overview of our main results, which appear in
Section~\ref{sec-results}.  See also Table~\ref{table-numbers} at the
end of the paper, where we list the counting sequences connected to
our results and conjectures.
\begin{itemize}
\item If $p$ is any one of the patterns $10, 001, 010, 011, 012$ then
  $\numpn=2^{n-1}$.

\item If $p$ is any one of the patterns $101, 0101,021$ then $\numpn$
  is the $n$-th Catalan number, and the ascent distribution on
  $\setpn$ is given by the Narayana numbers.  Moreover, the
  distribution of the bistatistic counting ascents and right-to-left
  minima is the same on $\seta_{021}(n)$ as it is on permutations of
  length $n$ avoiding~$132$.

\item If $p$ is any one of the patterns $102,0102,0112$ then
  $\numpn=(3^n+1)/2$.

\item Ascent sequences avoiding $101$ (or, equivalently, $0101$) are
  precisely the RGFs of the non-crossing partitions.

\item The set $\setpn$ consists solely of RGFs if and only if $p$ is a
  subpattern of $01012$.
\end{itemize}

The Catalan numbers, which we will frequently refer to, are given by
$C_n=\frac{1}{n+1}\binom{2n}{n}$. The Narayana numbers are given by
$N(n,k)=\frac{1}{n}\ch n,k,\ch n,k-1,$ and they refine the Catalan
numbers in that $C_n=\sum_k{N(n,k)}$.  It is well known that the
Narayana numbers record the distribution of the number of ascents on
permutations avoiding any given one of the patterns 132, 213, 231 and
312.

In addition to the results mentioned above we conjecture, in
Section~\ref{sec-conj}, avoidance sequences for the patterns $210$,
$0123$ and $0021$, in terms of entries in the {\sc\small Online
  Encyclopedia of Integer Sequences} \cite{oeis}, and also that $0021$
and $1012$ are \emph{Wilf equivalent}, that is, have the same
avoidance sequence.  In particular, we conjecture that ascent
sequences avoiding 210 are equinumerous with partitions avoiding so
called 3-crossings (a 2-crossing is simply a crossing in the usual
sense).

Moreover, we mention what little we know about growth rates of the
counting sequences for pattern avoiding ascent sequences.  In all
cases we know, or conjecture, the growth is exponential, as has been
proved for permutations avoiding any classical
pattern~\cite{marcus-tardos}.  Finally, we conjecture that
\emph{modified ascent sequences} (defined in \cite{bcdk}) avoiding 101
are in bijection with set partitions, and that there is a bijection
taking the number of non-ascents in such a modified sequence to the
number of blocks in the corresponding partition.  Note that this
conjecture would imply that modified ascent sequences avoiding 101
have super-exponential growth, that is, the number of such sequences
of length $n$ is not bounded by $C^n$ for any constant $C$.

\section{The results}\label{sec-results}

We first dispose of the easy results concerning patterns of length
less than three. There is, of course, no (nonempty) ascent sequence
avoiding the only pattern of length one.  There are three patterns of
length two, namely 00, 01 and 10.  It is easy to see that the only
ascent sequences avoiding 00 are the strictly increasing sequences
$0123\ldots$, one for each length $n$.  There is also precisely one
sequence of each length $n$ avoiding 01, namely the all zero sequence
$00\ldots0$.

The ascent sequences avoiding 10 are precisely the weakly increasing
sequences.  Since, in a weakly increasing ascent sequence, each
increase must be to a letter exceeding the previous maximum by one,
these sequences are determined by the places of their ascents.  Those
places can be chosen arbitrarily among the $n-1$ slots between
successive letters in a sequence of length $n$, so the number of such
ascent sequences is $2^{n-1}$ and the number of such sequences with
exactly $k$ ascents is $\binom{n-1}{k}$.  The same is true of three of
the sequences of length 3 as we now show.

\begin{thm}\label{thm-001}
  If $p$ is any one of the patterns 10, 001, 010 or 011, then
  $\numpn=2^{n-1}$.  Moreover, the number of sequences in $\setpn$
  with~$k$ ascents is $\binom{n-1}{k}$.
\end{thm}

\begin{proof}
We have already proved the claim for the pattern 10.  A proof for each
of the remaining three cases is easy to construct, given the following
characterizations of sequences avoiding each one of the patterns in
question:

\begin{itemize}
\item Ascent sequences avoiding 001 are precisely those that start
  with a strictly increasing sequence---necessarily $012\ldots k$ for
  some $k$---followed by an arbitrary weakly decreasing sequence of
  letters smaller than or equal to $k$.  An example is 01234444211.

\item Ascent sequences avoiding 010 are precisely those that are
  weakly increasing and thus of the form $00\ldots011\ldots122\ldots
  kk\ldots k$ for some $k$.

\item Ascent sequences avoiding 011 are precisely those that consist
  of a strictly increasing ascent sequence arbitrarily interspersed
  with 0's, for example 000100230400500.
\end{itemize}
In the first two cases we use the fact that the number of weakly
increasing (or weakly decreasing) sequences of nonnegative integers of
length $a$, and not exceeding $b$ in value, is $\binom{a+b}{a}$.
\end{proof}

It turns out that $\numa_{012}(n)$ is also $2^{n-1}$, but the ascent
distribution of sequences avoiding $012$ is different from the cases
in Theorem~\ref{thm-001}.

\begin{thm}\label{thm-012}
  We have $\numa_{012}(n)=2^{n-1}$.  The number of sequences in
  $\seta_{012}(n)$ with $k$ ascents is $\binom{n}{2k}$.
\end{thm}
\begin{proof}
Ascent sequences avoiding 012 are those that have no increasing
subsequence of length 3.  If an ascent sequence avoids 012, then,
after the initial 0, there can be no letters other than 0 and 1, since
any letter $a$ larger than 1 must be preceded by 1, in which case we
have a subsequence 01a, forming a 012.  Since the initial 0 can be
followed by any sequence of 0's and 1s, the sequences avoiding 012 are
precisely those that consist of an arbitrary binary string after the
initial 0, of which there are $2^{n-1}$.

To prove the result about the ascent distribution, suppose we are
given a $2k$-element subset $S=\{x_1<x_2<\cdots<x_{2k}\}$ of the set
$\{1,2,\ldots,n\}$, and we will use this to construct a binary string
of length $n$, starting with 0, with $k$ ascents.  The string to be
constructed will consist of 0's in all places preceding and including
place $x_1$, then the letters in places $x_{1}+1$ up to and including
$x_2$ will be 1s.  Generally, the letters in places $x_{2i}+1$ up to
and including $x_{2i+1}$, or after place $x_{2k}$, will be 0's and
those in places $x_{2i+1}+1$ up to and including $x_{2i+2}$ will be
1s.  The sequence thus constructed will clearly have ascents precisely
in the~$k$ places $x_{2i+1}$, where $i$ ranges from 0 to $k-1$.  It is
straightforward to construct the set $S$ of size $2k$ from a binary
string with $k$ ascents, which shows this is a bijection.
\end{proof}

Recall that two patterns $p$ and $q$ are Wilf equivalent if
$\numpn=\numa_q(n)$ for all $n$.  Theorems \ref{thm-001} and
\ref{thm-012} lead to the following result.

\begin{coro}\label{coro-10-wilf}
The patterns 10, 001, 010, 011 and 012 are Wilf equivalent.
\end{coro}

The following lemma is interesting for its own sake, as it
characterizes those patterns whose avoidance by a sequence $\bx$
guarantees that $\bx$ is an RGF, that is, a restricted growth function
that encodes a set partition.  Its second part also turns out to be
useful in proving some of our other results.

\begin{lemma}\label{lemma-rgf}
  Let $p$ be a pattern.  The $\setpn$ consists solely of RGF sequences
  if and only if~$p$ is a subpattern of 01012.  In particular, in a
  sequence $\bx$ avoiding any of these patterns, every occurrence of
  each letter $k\ge1$ is preceded by some occurrence of each of the
  letters $0,1,\ldots,k-1$.

\end{lemma}

\begin{proof}
Clearly the sequence 01013, not being RGF, must be excluded from any
set of RGF sequences.  This is guaranteed by the avoidance of $p$ only
if $p$ is a pattern in 01013, which is equivalent to being a
subpattern of 01012.  Thus, no other patterns than subpatterns of
01012 have avoidance sets consisting solely of RGFs.  We show that the
sequences avoiding any one of these patterns are RGFs, thereby
establishing the claim.  We prove the contrapositive, showing that if
an ascent sequence is not an RGF it must contain 10102.  That is
sufficient, since a sequence containing 10102 of course contains all
its subpatterns.

Suppose then that $\bx$ is an ascent sequence that is not an RGF.  Then
there must be a leftmost letter in $\bx$ that violates this.  If that
letter is $k$ then it is not preceded by $k-1$, but for some $i\ge2$
all the letters $0,1,\ldots,k-i$ appear in $\bx$ preceding $k$, and
their first appearances are in increasing order.  In order for an
ascent sequence to contain $k$ but not be preceded by $k-1$, it must
contain more ascents than those provided by the first occurrences of
each of $1,2,\ldots,k-i$.  Thus, one of the letters $1,2,\ldots,k-i$
appears at least twice as the rightmost letter in an ascent.  Suppose
this letter is $c$ and that the second one of these ascents is of the
form $\ldots ac\ldots$.  Thus, the first occurrence of $c$ must
precede this occurrence.  But, the first occurrence of $c$ is preceded
by the first occurrence of $a$.  Thus, we have a subsequence $\ldots
a\ldots c\ldots ac\ldots k\ldots$, and $acack$ is an occurrence of
$01012$.

The second part of the lemma follows directly from the definition of
RGFs.
\end{proof}

It turns out, as we will now show, that the ascent sequences avoiding
101 are the same as those avoiding 0101.  Moreover we will show that
these ascent sequences, which are RGFs according to
Lemma~\ref{lemma-rgf}, are precisely those that encode non-crossing
partitions.

\begin{thm}\label{thm-101}
  The ascent sequences avoiding $101$ are the same as those avoiding
  0101, and $\numa_{101}(n)=\numa_{0101}(n)=C_n$, the $n$-th Catalan
  number.  Moreover, the distribution of the number of ascents on
  these sequences is given by the Narayana numbers.
\end{thm}

\begin{proof}
  Clearly, an ascent sequence containing 0101 contains 101.  We show
  that the converse is also true, thereby showing that an ascent
  sequence avoids 101 if and only if it avoids 0101.  So, let $\ldots
  b\ldots a\ldots b\ldots$ be an occurrence of 101, so $a<b$.  By
  Lemma~\ref{lemma-rgf}, since 101 is a pattern in 01012, the first
  $b$ in the occurrence $\ldots b\ldots a\ldots b\ldots$ must be
  preceded by an occurrence of $a$.  This gives the subsequence
  $abab$, which is an occurrence of 0101.

We now exhibit a bijection from 101-avoiding ascent sequences of
length $n$ to 312-avoiding permutations of length $n$.  Given a
101-avoiding ascent sequence $\bx$, replace its 0's from left to right
with the numbers $k,(k-1),\ldots,2,1$, in decreasing order, where $k$
is the number of 0's in $\bx$.  Now repeat this for the 1's in $\bx$,
with the numbers $(\ell+k),(\ell+k-1),\ldots,(k+1)$, where $\ell$ is
the number of 1's in $\bx$ and so on.  For example, the sequence
01023200 is mapped to 45378621.  Suppose the resulting permutation $\pi$
contains an occurrence of 312, say in a subsequence $zxy$, where
$x<y<z$.  Then this subsequence in $\pi$ would correspond to a
subsequence in $\bx$ of the form $bab$ or $cab$, where $a<b<c$,
depending on whether $z=y+1$ or $z>y+1$.  In the first case $bab$ is
an occurrence of 101 in $\bx$.  In the second case the $c$ in $cab$
must be preceded by $b$, by Lemma~\ref{lemma-rgf}, again giving an
occurrence of 101 in $\bx$, a contradiction showing that this map
produces a 312-avoiding permutation.

This map is easily seen to be invertible.  Namely, given a permutation
$\pi=a_1a_2\ldots a_n$ avoiding 312, note that if $a_1=k$, then the
letters $k,k-1,\ldots,1$ must appear in this decreasing order in
$\pi$, for it to avoid 312.  Thus, the places of these letters are
filled with 0's in the ascent sequence $\bx$ corresponding to
$\pi$. Iterating this process we next find the leftmost letter $a_i$
in $\pi$ that is larger than $k$ and fill the places in $\bx$
corresponding to the places of the letters $a_i,a_i-1,\ldots,k+1$ in
$\pi$ by 1's, and so on.  It is straightforward to verify that the
sequence $\bx$ thus constructed avoids 101.

It is also easy to verify that this map preserves ascents.  Thus, the
101-avoiding ascent sequences are enumerated by the Catalan numbers,
and have ascent distribution given by the Narayana numbers, as is the
case for 312-avoiding permutations.
\end{proof}

\begin{thm}
The ascent sequences avoiding 0101 (equivalently, 101) are RGFs, and
as such they encode precisely all non-crossing partitions.
\end{thm}

\begin{proof}
Since 0101 is a pattern in 01012, the ascent sequences avoiding it are
RGFs, by Lemma~\ref{lemma-rgf}.

The occurrence of the pattern 0101 in an ascent sequence $\bx$ causes
a crossing, since it implies that we have letters $abab$, occurring in
that order, in $\bx$.  Thus, if the places where these four letters
occur are $x,y,z,w$, so that $x<y<z<w$, we have that $x$ and $z$ are
in the same block, and $y$ and $w$ in the same block, different from
the first one, which constitutes a crossing.

Conversely, suppose a set partition has a crossing consisting of
letters $a<b<c<d$, with $a$ and $c$ in the same block and $b$ and $d$
together in a different block.  Then either $a<b<c$ or $b<c<d$ gives
an occurrence of 101, which is equivalent to having an occurrence of
0101.
\end{proof}

In order to give a formula for the number of ascent sequences avoiding
one of the patterns in Theorem~\ref{thm-102-wilf} we need the
following lemma.

\begin{lemma}\label{lemma-tern}
  The number of ternary sequences of length $n$ on the letters $0,1,2$
  with an even number of 2's is $(3^n+1)/2$.
\end{lemma}

\begin{proof}
  It's easy to see that the number of ternary sequences of length $n$
  with exactly $k$ 2's is ${\binom{n}{k}2^{n-k}}$.  Thus, the
  difference between the number of such sequences with an even number
  of 2's and those with an odd number is given by
\[
(-1)^n\sum_{k}{\binom{n}{k}(-2)^{n-k}} = (-1)^n(1-2)^n = 1.
\] 
Since the total number of ternary sequences of length $n$ is $3^n$,
the claim follows.
\end{proof}

\begin{thm}\label{thm-102}
  We have $\numa_{102}(n)=(3^{n-1}+1)/2$.
\end{thm}

\begin{proof}
  We exhibit a bijection between sequences in $\seta_{102}(n)$ and
  ternary strings of length $n-1$ with an even number of 2's.  Together
  with Lemma~\ref{lemma-tern} this will prove the theorem.

  We will show that a 102-avoiding ascent sequence is necessarily
  composed of a weakly increasing ascent sequence of any length
  followed by a (possibly empty) sequence of what we call \emph{lifted
    binary strings}, namely, strings composed from two integers
  differing in size by one.  More precisely, 102-avoiding ascent
  sequences are those that can be written as
\[
x_1x_2\ldots x_k\; b_1^1b_2^1\ldots b_{i_1}^1\; b_1^2b_2^2\ldots
b_{i_2}^2\; \ldots b_1^mb_2^m\ldots b_{i_m}^m,
\]
where $x_1x_2\ldots x_k$ is a weakly increasing ascent sequence, the
$b_1^j$ form a strictly decreasing sequence with
$x_k>b_1^1>b_1^2>\cdots>b_1^m\ge0$ and each $b_p^q$ equals either
$b_1^q$ or $b_1^q+1$.  An example of such a sequence is
\[
0012234566 \; 56656 \; 4454 \; 2 \; 01001,
\]
where we have separated the respective lifted binary strings by
spaces, for clarity.

It is easy to see that a sequence of this type avoids 102.  We now
explain why any 102-avoiding sequence has this structure.

Any ascent sequence $\bx$ starts with a weakly increasing sequence
$x_1x_2\ldots x_k$ of length at least 1.  In what follows, we choose
$k$ such that this initial weakly increasing sequence is of maximal
length. By Lemma \ref{lemma-rgf}, since 102 is a pattern in 01012, all
the integers between 0 and the maximum letter in $\bx$ must appear in
$\bx$.

If there are no letters following this initial weakly increasing
subsequence of $\bx$ then $\bx$ clearly has the form described, with
an empty sequence of lifted binary strings after the initial weakly
increasing sequence. Otherwise there is a descent immediately
following the initial sequence $x_1x_2\ldots x_k$, that is, $x_k$ is
followed by $b_1^1$ where $x_k>b_1^1$.  If any letter $x$ after
$b_1^1$ satisfied $x>b_1^1+1$, then, for some~$i$, the subsequence
$x_i,b_1^1,x$ would form a 102.  This is because the letter
$(b_1^1+1)\le x_k$ must appear in the initial sequence $x_1x_2\ldots
x_k$.

Thus, if $x_k$ is not the last letter of $\bx$, then $x_k$ is followed
by a sequence $b_1^1b_2^1\ldots b_{i_1}^1$, each of whose letters is
either $b_1^1$ or $(b_1^1+1)$, and where $b_1^1<x_k$.  An analogous
argument now shows that $b_{i_1}^1$ must be followed by a sequence
$b_1^2b_2^2\ldots b_{i_2}^2$ of letters $b_1^2$ and $b_1^2+1$, where
$b_1^2 < b_1^1$, and so on.

We now encode a 102-avoiding ascent sequence with a ternary string
$\bt=t_2t_3\ldots t_n$, of length $n-1$ on the symbols 0, 1, 2, with
an even number of 2's.  Note that the numbering of the letters in
$\bt$ starts with 2.  Recall that $k$ is the index of the last copy of
the largest letter in $\bx$, that is, the last letter of the initial
weakly increasing sequence.  In what follows we refer to the prefix
$x_1x_2\ldots x_k$ of $\bx$ as its first part, and the remainder as
its second part, and likewise for the corresponding parts of $\bt$.
We let each of the letters $t_i$ in the first part of $\bt$ be a 0 if
$x_i=x_{i-1}$, and 1 otherwise, that is, if $x_{i}=x_{i-1}+1$.  Some
of the 1s in this first part of $\bt$ will later be changed to 2's.
We refer the reader to the example after this proof for clarification
of this and the rest of the proof.

The rest of $\bt$ is constructed from what remains of $\bx$, that is,
from the tail $x_{k+1}\ldots x_n$.  This is the part of $\bx$ that
consists of lifted binary strings with entries that decrease in size
between two consecutive such strings.  Here, we let $t_i$ be a 2 if
$x_i$ is the first letter $\ell$ in one of these lifted binary
strings, a 0 if $x_i=\ell$ but is not the first in the string, or a 1
if $x_i=\ell+1$ in the string.

We now change some of the 1s in the first part of $\bt$ to 2's as
follows: Suppose the leftmost 2 in (the second part of) $\bt$ is
$t_i$.  Then the integer $x_i+1$ occurs in the first part of $\bx$,
and we change $t_j$ from 1 to 2, where $x_j$ is the leftmost
occurrence of $x_i+1$ in $\bx$.

The inverse of this map is described as follows, where we construct a
102-avoiding ascent sequence of length $n$ from a ternary sequence of
length $n-1$, with $2k$ copies of 2's: Let~$t_e$ be the $(k+1)$-st 2
in $\bt$.  Then, for $2\le i<e$, $x_i=x_{i-1}$ if $t_i=0$, and
$x_i=x_{i-1}+1$ if $t_i$ is 1 or 2.  We then let $x_e$ be $x_j-1$
where $x_j$ is the $k$-th 2 in $\bt$, that is, the last 2 in the first
part of $\bt$.  The binary string of 0's and 1s strictly between $x_e$
and $x_d$, where $x_d$ is the next 2 in $\bt$ after $t_e$, now
determines the letters $x_{e+1}$ through $x_{d-1}$, where $t_i=0$
means $x_i=x_j-1$ and $t_i=1$ means $x_i=x_j$.  This is now repeated
for the remainder of $\bt$: $x_d$ is set to $x_m-1$, where $t_m$ is
the $k-1$-st 2 in $\bt$, and the binary string between $t_d$ and the
next 2 in $\bt$ determines the corresponding letters in $\bx$ in a way
analogous to the previous case.  This is then repeated for the rest of
$\bt$.
\end{proof}

Here is an example of how the bijection in the proof of
Theorem~\ref{thm-102} works:
\begin{align*}
\phantom{0}
\bt=~\,\;2\;1\;0\;0\;1\;2\;1\;0\;2\;2\;2\;1\;0\;0\;2\;0\;1\;2\;2\\ \bx=0\;1\;2\;2\;2\;3\;4\;5\;5\;6\;7\;6\;7\;6\;6\;5\;5\;6\;3\;0
\end{align*}
For example, the letter below the penultimate 2 in $\bt$ is $3=4-1$,
because the second~2 in $\bt$ is in a place where $\bx$ has a 4.
Since the third 2 in $\bt$ has a 6 below it in $\bx$, the letter below
the third last 2 in $\bt$ is $6-1=5$.  The letters between the 3 and
the 5 in $\bx$ mentioned here are 56, corresponding to the 01 in the
corresponding part of $\bt$.

We next show that the patterns 102, 0102 and 0112 are Wilf equivalent.

\begin{thm}\label{thm-102-wilf}
  The patterns 102, 0102 and 0112 are Wilf equivalent.  In particular,
  we have $\seta_{102}(n)=\seta_{0102}(n)$.
\end{thm}
\begin{proof}
  Of course, an occurrence of 0102 implies an occurrence of 102. It
  suffices to show that the converse is also true.  Let $b$ be a
  letter involved in an occurrence of 102 in an ascent sequence $\bx$,
  and let $a$ and $c$ be letters such that $bac$ is such an
  occurrence, so $a<b<c$.  Now, 102 is a subpattern of 01012, so, by
  Lemma \ref{lemma-rgf}, $b$ must be preceded by~$a$, and we therefore
  have $abac$ as a subsequence in $\bx$, constituting an occurrence of
  0102.

  We next show that the number of 0112-avoiding ascent sequences of
  length $n$ is given by $(3^{n-1}+1)/2$, which, together with
  Theorem~\ref{thm-102}, completes the proof.

  We claim that the ascent sequences avoiding 0112 are precisely those
  that consist of a strictly increasing sequence $012\ldots k$
  followed by a weakly decreasing sequence, the entire sequence
  arbitrarily interspersed with~0's.  It is clear that a strictly
  increasing sequence followed by a weakly decreasing sequence cannot
  contain a 0112, and also that interspersing such a sequence with 0's
  cannot create a 0112.  We thus only need to show that an occurrence
  of 0112 in an ascent sequence prevents it from having this
  prescribed form.  That is straightforward, since the letter
  corresponding to the second 1 in such an occurrence could only
  belong to the weakly decreasing sequence, but that prevents it from
  being followed by a larger letter constituting the 2.

  To count the sequences described in the previous paragraph, observe
  that such a sequence is either the all zero sequence or else it
  consists, apart from the interspersed zeros, of a strictly
  increasing sequence $1,2,\ldots i$ for some $i\ge1$, followed by a
  weakly decreasing sequence of letters $a_j$ such that $1\le a_j\le
  i$. If there are $n-1-k$ zeros after the initial 0 in such a
  sequence $\bx$, then we can pick their positions in $\ch
  n-1,n-1-k,=\ch n-1,k,$ ways and there are $k$ places left for the
  increasing and weakly decreasing sequence.  Thus, the length of the
  increasing sequence can range from 1 to $k$.  If the length of the
  increasing sequence is~$i$ then the weakly decreasing sequence has
  length $k-i$ and its letters are positive and not exceeding $i$.  As
  pointed out in the proof of Theorem~\ref{thm-001}, the number of
  weakly decreasing sequences with such parameters is $\ch
  (k-i)+(i-1),k-i, =\ch k-1,k-i,$.  Thus, the total number of ascent
  sequences avoiding 0112 is
\[
1+\sum_{k=1}^{n-1}{\ch n-1,k,\sum_{i=1}^{k}{\ch k-1,k-i,}},
\]
which easily simplifies to $(3^{n-1}+1)/2$.
\end{proof}

We now show that ascent sequences avoiding 021 are counted by the
Catalan numbers, but first a definition of a set of sequences that
turn out to be equinumerous, for each $n$, with $\seta_{021}(n)$.

\begin{defn}
  A \emph{restricted ascent sequence} is an ascent sequence
  $x_1x_2\ldots x_n$\linebreak with $x_i\ge m_i-1$ for all $i>0$, where $m_i$ is
  the maximum value among $x_1,x_2,\ldots,x_{i-1}$.  We denote the set
  of such sequences of length $n$ by $\clr_n$.
\end{defn}

\begin{thm}\label{thm-021}
  There is a bijection preserving ascents between $\clr_n$ and
  $\seta_{021}(n)$, for all~$n$.  Consequently, $\numa_{021}(n)=C_n$,
  the $n$-th Catalan number.
\end{thm}
\begin{proof}
  In a restricted ascent sequence, the only letters that can appear
  between two successive LR-maxima $x_i$ and $x_k$ or after its
  rightmost LR-maximum $x_i$ are $x_i-1$ and $x_i$.  Moreover, any
  ascent sequence satisfying this condition is a restricted ascent
  sequence.

  On the other hand, an ascent sequence $\bx$ is easily seen to avoid
  021 if and only its nonzero entries are weakly increasing. Thus, in
  particular, $\bx$ is such a sequence if and only if it has just two
  kinds of letters, namely $x_i$ and 0, between successive LR-maxima
  $x_i$ and $x_k$, and after its rightmost LR-maximum $x_i$.

  Interchanging $x_i-1$ and 0 between successive LR-maxima in the
  sequences described in each of the previous two paragraphs clearly
  defines a bijection preserving ascents between the two sets of
  sequences.  It is shown in Theorem 9 (due to Hilmar Gudmundsson) in
  \cite{kitaev-remmel-2plus2} that $|\clr_n|=C_n$, which completes the
  proof.
\end{proof}

Theorems~\ref{thm-101} and \ref{thm-021} now imply the following.

\begin{coro}\label{coro-101-wilf}
The patterns 101, 0101 and 021 are Wilf equivalent.
\end{coro}

An example of the bijection mentioned in the last paragraph of the
proof of Theorem~\ref{thm-021} is given by the following two
sequences, where the top one is a restricted ascent sequence, and the
bottom one a sequence avoiding 021, and where the only differences
between the two sequences are in the places of the zeros in the bottom
sequence (apart from the initial zero in each):
\begin{align*}
0\;1\;2\;3\;2\;3\;4\;4\;3\;4\;6\;5\\ 0\;1\;2\;3\;0\;3\;4\;4\;0\;4\;6\;0
\end{align*}
Our goal is to give a bijection from the set of 021-avoiding ascent
sequences to 132-avoiding permutations, and it should be noted that
021 and 132 are equivalent as patterns.  Since there are very
transparent bijections between restricted and 021-avoiding ascent
sequences, as shown above, and between 132-avoiding and 231-avoiding
permutations, by reversing each permutation, we choose to exhibit a
bijection between restricted ascent sequences and 231-avoiding
permutations, which is more convenient.  First a definition and a
lemma.

\begin{defn}\label{def-maximal}
  A letter $x_i$ in an ascent sequence $x_1x_2\ldots x_n$ is
  \emph{maximal}\/ if $x_i$ is as large as it can be for its place,
  that is, if $x_i=\asc(x_1x_2\ldots x_{i-i})+1$.  In particular, the
  initial zero in an ascent sequence is defined to be maximal.  If a
  maximal letter $x_i$ satisfies $x_i=x_{i+1}=x_{i+2}=\cdots=x_{i+k}\ne
  x_{i+k+1}$, or if $i+k=n$, for some $k\ge1$, then we say that $x_i$
  is a \emph{repeated maximal letter}, and that $x_{i+k}$ is its
  \emph{last repetition}.  A maximal letter that is not repeated is
  its own last repetition.
\end{defn}
For example, the maximal letters in the following ascent sequence are
in bold and the last repetitions of repeated maximal letters are
overlined:
${\mathbf0}\,0\,{\ol0}\,{\mathbf1}\,0\,1\,2\,0\,{\mathbf4}\,
\ol4\,2\,3\,2\,0\,{\mathbf6}\,4\,1\,4\,{\mathbf8}\,8\,{\ol8}\,5$.

\begin{lemma}
  In a restricted ascent sequence $\bx=x_1x_2\ldots x_n$, let $x_i$ be
  the last repetition of the rightmost maximal letter in $\bx$, and
  let $\bx'=x_{i+1}x_{i+2}\ldots x_n$.  Then, if $\bx'$ is nonempty,
  we must have $x_{i+1}=x_i-1$ and the sequence obtained by
  subtracting $x_{i+1}$ from each letter in $\bx'$ is a restricted
  ascent sequence.
\end{lemma}
\begin{proof}
  Since $x_i$ is the last repetition of the rightmost maximal letter,
  $x_{i+1}$ must be smaller than $x_i$.  Since $\bx$ is a restricted
  ascent sequence, $x_{i+1}$ cannot be smaller than $x_i-1$.

  Now, if $k>i+1$ then $x_k<x_i+1+a$, where $a$ is the number of
  ascents in $\bx'$ that strictly precede $x_k$, for otherwise $x_k$
  would be a maximal letter, contrary to the assumption about~$x_i$.
  Equivalently, $x_k\le x_{i+1}+1+a$.  Also, $x_k\ge x_i-1= x_{i+1}$,
  since $\bx$ is a restricted ascent sequence.  Thus, subtracting
  $x_{i+1}$ from each letter in $\bx'$ produces a sequence of
  nonnegative integers that satisfies the condition defining ascent
  sequences.  The resulting sequence is a restricted ascent sequence
  because subtracting the same number from each letter doesn't affect
  the condition for restricted ascent sequences.
\end{proof}

As an example, the last repetition of the rightmost maximal letter in
the restricted ascent sequence $\bx=00101332232434665$ is the second
3, so $\bx'=2232434665$.  Reducing each letter in $\bx'$ by 2 we
obtain 0010212443, which is a restricted ascent sequence.

We now define a bijective map $\phi:\clr_n\rightarrow\clp_n(231)$,
where $\clp_n(231)$ is the set of 231-avoiding permutations of
length~$n$.  Because we construct $\pi=\phi(\bx)$ recursively we need
to define an auxiliary map $\om$, which takes two arguments, a
restricted ascent sequence $\by$ and an interval of integers whose
length equals that of $\by$.  We then set $\phi(\bx)=\om(\bx,[1,n])$,
where~$n$ is the length of $\bx$.

The map $\om$ is called recursively on three segments into which $\bx$
is partitioned, each of which is fed to $\om$ together with a part of
the original interval.  We let $\epsilon$ stand for the empty string
and the empty sequence, and set $\om(\epsilon,\emptyset)=\epsilon$.
Given a nonempty restricted ascent sequence $\bx$, let $m$ be its
rightmost maximal letter.  Then $\bx$ can be expressed uniquely as the
concatenation $LmR$, where $L$ is the part of $\bx$ preceding the last
repetition of $m$, and $R$ the part following that $m$.  We let
$\ol{R}$ be the restricted ascent sequence obtained by subtracting the
first letter of $R$ from each of its letters.

Suppose the last repetition of the rightmost maximal letter $m$ is in
place $p$ in $\bx$.  Intuitively,~$\phi$, via $\om$, writes, in the
first place in $\pi$, the number 1 if $m$ is a repeated maximal
letter, but $p$ otherwise.  That first letter in $\pi$ is then
followed by $\om(L,I)$, and then by $\om(R,J)$, where $I$ and $J$ are
the appropriate intervals.  Formally, we define $\om$ as follows when
$\bx=LmR$, $[a,b]$ is an interval of integers whose length equals the
length of $\bx$, $\ell$ is the length of $L$ and $\oplus$ is
concatenation: \def\op{\oplus}
\[
\om(LmR,[a,b])= 
\begin{cases}
  a\op\om(L,[a+1,a+\ell])\op\om(\ol{R},[a+\ell+1,b]),& \mbox{~if
    $m$ is repeated},\\[1ex]
  (a+\ell)\op\om(L,[a,a+\ell-1])\op\om(\ol{R},[a+\ell+1,b]),&
  \mbox{~otherwise}.
             \end{cases}
\]
Here is an example of how $\phi$ works, via $\om$.  Note that the
maximal letters in 011213232 are the leftmost occurrences of 0, 1, 2
and 3, with the second 1 the only repeated maximal letter:
\begin{align*}
&\phi(011213232)=\om(011213232,[1,9])=\\[1ex]
&6\op\om(01121,[1,5])\op\om(\ol{232},[7,9])=\\[1ex]
&6\op4\op\om(011,[1,3])\op\om(\ol{1},[5,5])\op\om(010,[7,9])=\\[1ex]
&6\op4\op\om(011,[1,3])\op\om(0,[5,5])\op8\op\om(0,[7,7])\op\om(0,[9,9])=\\[1ex]
&6\op4\op1\op\om(01,[2,3])\op\om(\emptyset,\emptyset)\op\om(0,[5,5])\op8\op\om(0,[7,7])\op\om(0,[9,9])=\\[1ex]
&6\op4\op1\op3\op\om(0,[2,2])\op\om(\emptyset,\emptyset)\op\om(\emptyset,\emptyset)\op\om(0,[5,5])\op8\op\om(0,[7,7])\op\om(0,[9,9])=\\[1ex]
&6\op4\op1\op3\op2\op\emptyset\op\emptyset\op5\op8\op7\op9=641325879.
\end{align*}


\begin{thm}\label{thm-021-narayana}
  The map $\phi$ described above is a bijection from $\clr_n$ to
  $\clp_n(231)$ and $\asc(\bx)=\des(\phi(\bx))$ for any
  $\bx\in\clr_n$.  Consequently, the distribution of the number of
  ascents on $\seta_{021}(n)$ (or, equivalently, $\clr_n)$ is given by
  the Narayana numbers.
\end{thm}

\begin{proof}
  We first prove the claim $\asc(\bx)=\des(\phi(\bx))$, then show that
  $\phi$ produces a 231-avoiding permutation, and finally show that
  $\phi$ is bijective.  By Theorem \ref{thm-021} the claim about
  ascents applies equally to sequences in $\clr_n$ and
  $\seta_{021}(n)$.

  It follows by induction from the definition of $\phi$, via the
  definition of $\om$, that $\phi$ yields a permutation of the
  letters~1 through $n$, where $n$ is the length of $\bx$.  When $\om$
  writes out a letter $f$, that letter is followed by two words,
  constructed from $L$ and $R$.  In the case when the rightmost
  maximal letter $m$ of $\bx$ is repeated, and so is not the second
  letter in an ascent in $\bx$, $f$ will be smaller than all the
  letters eventually following it in $\pi$, thus not causing a descent
  in $\pi$.  If $m$ is not repeated, and is not the first letter in
  $\bx$, then $m$ is the second letter in an ascent in $\bx$, and the
  letter $f=a+\ell$ written out by $\om$ is greater than all the
  letters arising from applying $\om$ to $L$, and will thus be the
  first letter in a descent in $\pi$.  Thus $\phi$, via $\om$,
  translates ascents in $\bx$ into descents in $\pi$, and non-ascents
  into non-descents.

  Analyzing the two cases in the previous paragraph shows that a
  letter $f$ written out by~$\om$ at any stage can not be the 2 in an
  occurrence of the pattern 231 in $\pi$, since $f$ is either smaller
  than everything following it, or else followed by letters smaller
  than $f$ that then are followed by letters larger than $f$.  Since
  no letter in $\pi$ can be the 2 in an occurrence of 231, there can
  be no occurrence of 231 in $\pi$, so $\phi$ is indeed a map to
  $\clp_n(231)$.

  We now show that $\phi$ is injective.  This is trivially true for
  $n=1$.  Suppose $\bx=LmR$ and $\bx'=L'm'R'$ are restricted ascent
  sequences of length at least 2, where $m$ and $m'$ are the
  respective last repetitions of their rightmost maximal letters.  If
  $\pi=\phi(\bx)=\phi(\bx')$ then either both $m$ and $m'$ are
  repeated maximal letters, in which case $\pi(1)=1$, or neither is.
  If neither is repeated then $\pi(1)=p>1$, where $m$ and $m'$ are in
  place $p$ in $\bx$ and $\bx'$, respectively.  Thus, we must have
  that $\om(L,I)=\om(L',I)$, where $I$ is the interval passed with $L$
  and $L'$ to $\om$, and likewise for $R,R'$ and the interval $J$
  passed with them to $\om$.  By induction, since $L,L',R,R'$ are all
  strictly shorter than $\bx$, we may assume that this implies that
  $L=L'$ and $R=R'$.  That, in turn, implies that $\bx=\bx'$.

  If $m$ and $m'$ are repetitions of the rightmost maximal letter in
  $\bx$ and $\bx'$, respectively, then $\pi(1)=1$.  In fact, $\pi$
  then starts with $12\ldots k$ where $k$ is the number of repetitions
  of the rightmost maximal letters $m$ and $m'$, which must therefore
  be repeated the same number of times.  Once $\om$ has been applied
  successively to these repetitions we are back in the case where $m$
  and $m'$ are not repeated and, by an argument identical to the one
  applied to that case above, we can infer that $\bx=\bx'$.

  Since $\phi:\clr_n\rightarrow\clp_n(231)$ and $\phi$ is injective,
  and we know that $\clr_n$ and $\clp_n(213)$ are equinumerous, the
  map $\phi$ is a bijection.
\end{proof}

There is an easy, and well known, bijection from 231-avoiding
permutations to non-crossing partitions, that is, set partitions of
$\{1,2,\ldots,n\}$ that do not contain $a<b<c<d$ with $a$ and $c$ in
one block and $b$ and $d$ in another.  Namely, given any permutation,
split it into blocks after each ascent.  As an example, the
231-avoiding permutation 641325879 is split into the blocks
641-32-5-87-9, or 146-23-5-78-9 in our standard notation.

It is straightforward to check that if the permutation avoids 231,
then the resulting partition will have its blocks in order of
increasing minima.  Also, a crossing in the resulting partition would
imply that there were letters $a<b<c<d$ appearing in the corresponding
permutation in the order $cadb$ or $dbca$, but then $cdb$ or $bca$
would have formed the pattern 231 in that permutation.  With this in
hand it is easy to show that when restricted to 231-avoiding
permutations, this map is a bijection to non-crossing partitions.

Composing this bijection with the bijection described before
Theorem~\ref{thm-021-narayana} gives a bijection between non-crossing
partitions and ascent sequences avoiding 021. Using the same examples
as before, we see that the 021-avoiding ascent sequence
$\bx=011213232$ corresponds to the non-crossing partition
146-23-5-78-9.

\section{Further remarks, conjectures and open
  problems}\label{sec-conj}

For permutations, it was conjectured indepently by Stanley and Wilf,
and proved by Marcus and Tardos \cite{marcus-tardos}, that the number
of permutations of length $n$ avoiding any single pattern $p$ is
bounded by $C^n$ for some constant $C=C(p)$ depending only on
$p$. Define the \emph{growth rate} of a sequence $x_0,x_1,x_2,\ldots$
to be the limit $\displaystyle\lim_{n\rightarrow\infty}(x_n)^{1/n}$,
if it exists. It is known that this limit always exists in the case of
single patterns, so if $C(p)$ above is chosen as small as possible it
is the growth rate in that case.  Some work has recently been done on
finding the possible growth rates of permutation classes, that is, the
growth rates of sets of permutations avoiding sets of patterns,
yielding very interesting results on what growth rates are possible.
See \cite{klazar} for a general overview, and \cite{vatter} for the
most recent results on permutation classes.

From the results we have presented here (see Table
\ref{table-numbers}), it is straightforward to show that the growth
rates for 001, 102 and 101 are 2, 3 and 4, respectively.  Moreover, if
Conjecture~\ref{conj-210} below is true, the growth rate for 210 is 9,
as can be derived from Proposition~1 in~\cite{mbm-xin}.  Also, should
Conjecture~\ref{conj-0123} hold, we would have that the growth rate
for 0123 was approximately $3.247$, as can be computed from the linear
recurrence $a(n)=5a(n-1)-6a(n-2)+a(n-3)$ given for sequence {\sc\small
  A080937} in \oeis~\cite{oeis}.  Moreover, Conjecture~\ref{conj-0021}
implies a growth rate of 5 for 0021 (and 1012), given the recursive
formula for sequence {\sc\small A007317} in \cite{oeis}.

It is a tempting conjecture that the Marcus-Tardos Theorem mentioned
above also holds for ascent sequences, but this requires more work to
be justified.

By reversing a 231-avoiding permutation we obtain a 132-avoiding
permutation, and turn each descent into an ascent, and conversely.
This, together with Theorem~\ref{thm-021-narayana}, shows that the
number of ascents has the same distribution on 021-avoiding (a.k.a.\
132-avoiding) ascent sequences and 132-avoiding permutations.  We
conjecture that this result can be strengthened as follows.  

\begin{conj}\label{conj-bi-021}
The bistatistic $(\asc,\rlmin)$ has the same distribution on
$\seta_{021}(n)$ and on 132-avoiding permutations of length $n$.
\end{conj}
After a preliminary version of this paper was posted on the internet
Bruce Sagan proved Conjecture~\ref{conj-bi-021} (personal
communication).

Let $\fwd$ be the length of the maximal final weakly decreasing
sequence in an ascent sequence.  For example,
$\fwd(01123035523220)=4$, since 3220 has length 4.  
are the rightmost occurrences of $0,2,3$ and $5$.  
We then have the following conjecture, where $\zeros(\bx)$ is the
number of 0's in an ascent sequence $\bx$.

\begin{conj}\label{conj-0012}
  We have $\numa_{0012}(n)=C_n$, the $n$-th Catalan number.  Moreover,
  the bi\-statistic $(\asc,\fwd)$ on $\seta_{0012}(n)$ has the same
  distribution as $(\asc,\rlmax)$ does on permutations avoiding the
  pattern 132.  In particular, this implies that the number of ascents
  has the Narayana distribution on $\seta_{0012}(n)$.  Also, the
  bistatistics $(\asc,\fwd)$ and $(\asc,\zeros)$ have the same
  distribution on 0012-avoiding ascent sequences.
\end{conj}

We should point out that the bistatistics $(\asc,\rlmin)$ on
132-avoiding permutations and $(\asc,\rlmax)$ on 231-avoiding
permutations mentioned in Conjectures \ref{conj-bi-021} and
\ref{conj-0012} have the same distribution as double upsteps and the
number of returns to the $x$-axis on Dyck paths.  We have found more
apparent equidistributions of bistatistics of a similar kind to those
in Conjectures~\ref{conj-bi-021} and~\ref{conj-0012}, and it seems
certain that many more could be found easily.

\begin{conj}\label{conj-210}
  The number $\numa_{210}(n)$ equals the number of non-3-crossing set
  partitions of $\{1,2,\ldots,n\}$.  See sequence A108304 in
  \cite{oeis}, where non-3-crossing is also defined.
\end{conj}

\begin{conj}\label{conj-0123}
  The number $\numa_{0123}(n)$ equals the number of Dyck paths of
  semilength $n$ and height at most 5.  See sequence A080937 in
  \cite{oeis}.
\end{conj}

\begin{conj}\label{conj-0021}
The patterns 0021 and 1012 are Wilf equivalent, and
$\numa_{0021}(n)=\numa_{1012}(n)$ is given by the binomial transform
of Catalan numbers, which is sequence A007317 in~\cite{oeis}.
\end{conj}

It is easy to show that the Wilf equivalences proved or conjectured
here are the only ones possible for the set of all patterns of length
at most four.  This is because other pairs of sequences are seen, by
computer testing, to diverge.

Thus, the following are all the Wilf equivalences for patterns of
length at most four, where the cases of 0012 and the pair
$(0021,1012)$ depend on Conjectures~\ref{conj-0012} and
\ref{conj-0021}, respectively, and the other results on
Corollaries~\ref{coro-10-wilf} and \ref{coro-101-wilf} and
Theorem~\ref{thm-102-wilf}, in addition to the trivial equivalence of
00 and 01:
\begin{align*}
&00\sim 01,\\ &10\sim 001\sim 010\sim 011\sim 012,\\ &102\sim 0102\sim
  0112,\\ & 101\sim 021\sim 0101\sim 0012,\\ &0021\sim 1012.
\end{align*}

When computing the numbers of ascent sequences avoiding various pairs
of patterns of lengths 3 and 4 we found several sequences of numbers
recognized by the \oeis\ \cite{oeis}, such as Motzkin numbers, various
transformations of the Catalan and Fibonacci numbers, and sequences
counting permutations avoiding pairs of patterns. We also found some
apparent Wilf equivalences that might be interesting.  We do not list
any of these here, but we will gladly share our data with anybody who
might be interested.

Finally, we consider the \emph{modified ascent sequences} defined in
\cite[Section~4.1]{bcdk}.  Given an ascent sequence $\bx$, we
successively create sequences $\bx_1,\bx_2,\ldots,\bx_k$, where $k$ is
the number of ascents in $\bx$.  The sequence $\bx_{i+1}$ is
constructed from $\bx_i$ by increasing by 1 each letter in $\bx_i$
that precedes the $i$-th ascent in $\bx$ and is larger than or equal
to the larger letter in that ascent.  The sequence $\bx_k$ is then
defined to be the modified ascent sequence associated to $\bx$.
Observe that in constructing each $\bx_i$ we preserve the ascents in
$\bx$.  It is easy to see that this is an invertible process so that
modified ascent sequences are in bijection with ascent sequences.  As
an example, the modified ascent sequence of 010221212 successively
becomes 010331212 and 010441312, with the first two ascents causing no
changes.

\begin{conj}\label{conj-modi}
  On modified ascent sequences the patterns $101,0101,1021,1102,1120$,
  $1210$ are all Wilf equivalent, and the number of ascent sequences
  avoiding any one of these patterns is the $n$-th Bell number.  These
  sequences are thus equinumerous with partitions of an $n$-element
  set.  Moreover, the distribution of the number of ascents on such
  sequences is the reverse of the distribution of the number of blocks
  on set partitions.  That is, the number of modified ascent sequences
  of length $n$ with $k$ ascents, and avoiding any single one of these
  patterns, equals the number of set partitions of an $n$-element set
  with $n-k$ blocks.
\end{conj}
Bruce Sagan (personal communication) has proved the above conjecture
in the case of the pattern 101 and $k=1$.

Note that if Conjecture~\ref{conj-modi} is true then the number of
modified ascent sequences avoiding $101$ would not have exponential
growth, since that is known to be false for the Bell numbers; they
grow faster than $C^n$ for any constant $C$.

\def\mp#1{\begin{minipage}[c]{10mm}\kern1ex\centering#1\\[-.5em]\phantom{a}
\end{minipage}}

\def\mb#1{\begin{minipage}[c]{15em}\kern2ex\centering#1\\[-.5em]\phantom{a}\end{minipage}}

\begin{table}[ftbp]
\small \centering
\begin{tabular}{|c|l|c|c|l|}
\hline Pattern $p$ & Number of sequences avoiding $p$ & {\oeis} & Formula & Reference\\ \hline \hline

\hline
\mp{001\\010\\ 011\\012}
&  \begin{minipage}[c]{15em}$1,2,4,8,16,32,64,128,256,512,\ldots$
   \end{minipage}
&A000079 & $2^{n-1}$ &Thms.~\ref{thm-001}, \ref{thm-012}\\

\hline \mp{102\\ 0102\\ 0112} & $1,2,5,14,41,122,365,1094,3281,9842,\ldots$ & A007051 & $(3^n+1)/2$ & Thm.~\ref{thm-102}\\

\hline
\mp{101\\021\\0101}
&  \mb{$1,2,5,14,42,132,429,1430,4862,16796,\ldots$}

& A000108 & $\displaystyle\frac{1}{n+1}\binom{2n}{n}$ & Thms.~\ref{thm-101}, \ref{thm-021}\\

\hline \mp{000} & \mb{$1,2,4,10,27,83,277,1015,4007,17047,$\\ $77451,374889,1923168,10427250,\ldots$} & & &\\

\hline \mp{100} & \mb{$1,2,5,14,44,153,583,2410,10721,50965,$\\ $257393,1374187,7722862,45520064,\ldots$} & & &\\

\hline \mp{110} & \mb{$1,2,5,14,43,143,510,1936,7774,32848,$\\ $145398,671641,3227218,16084747,\ldots$} & & &\\

\hline \mp{120} & \mb{$1,2,5,14,42,133,442,1535,5546,20754,$\\ $80113,317875,1292648,5374073,\ldots$} & & &\\

\hline \mp{201} & \mb{$1,2,5,15,52,201,843,3764,17659,86245,$\\ $435492,2261769,12033165,65369590,\ldots$} & & &\\

\hline \mp{210} & \mb{$1,2,5,15,52,202,859,3930,19095,97566,$\\ $520257,2877834,16434105,965054901,\ldots$} & A108304 & & Conj.~\ref{conj-210}\\

\hline \mp{0123} & \mb{$1,2,5,14,42,131,417,1341,4334,14041,$\\ $45542,147798,479779,1557649,\ldots$} & A080937 & & Conj.~\ref{conj-0123}\\

\hline \mp{0021\\ 1012} & \mb{$1,2,5,15,51,188,731,2950,12235,51822,$\\ $223191,974427,4302645,19181100\ldots$} & A007317 & & Conj.~\ref{conj-0021}\\

\hline
\end{tabular}
\caption{\label{table-numbers}Number sequences for pattern avoidance by ascent sequences. All patterns of length 3 are listed, a few of length 4. \oeis\ refers to entry in \cite{oeis}.}
\end{table}

\section{Acknowledgments}

We are grateful to Anders Claesson, Mark Dukes, Sergey Kitaev and Jeff
Remmel for useful discussions and suggestions.  We are also deeply
indebted to a referee who corrected a few errors and made many useful
suggestions, which led to a substantial improvement of the
presentation.  In particular, that referee pointed out the bijection
now used in the proof of Theorem~\ref{thm-101}.


\end{document}